\documentclass[12pt]{article}
\usepackage{amssymb,amsmath,amsthm,color}
\numberwithin{equation}{section} %
\theoremstyle{theorem}
\newtheorem{thm}{Theorem}[section]
\newtheorem{lem}{Lemma}[section]

\newtheorem{rem}{Remark}[section]

\newtheorem{ex}{Example}[section]
\newtheorem{prop}{Proposition}[section]

\numberwithin{equation}{section}

\begin{document}
\title{Wave propagation and its stability for a class of discrete diffusion systems}
\author{
Zhixian Yu\footnote{The corresponding author. College of Science, University of Shanghai for Science and Technology,
Shanghai, 200093, China. Email: zxyu0902@163.com. Partially supported by Natural Science Foundation of Shanghai (No.18ZR1426500).}, Yuji Wan\footnote{College of Science, University of Shanghai for Science and Technology,
Shanghai, 200093, China. Email: yjwan0530@163.com} and Cheng-Hsiung Hsu\footnote{Department of Mathematics, National Central University, Chung-Li 32001, Taiwan. Email: chhsu@math.ncu.edu.tw. Partially supported by the MOST and NCTS of Taiwan.}}
\date{}
\maketitle
\begin{abstract}
This paper is devoted to study the wave propagation and its stability for a class of two-component discrete diffusive systems.
We first establish the existence of positive monotone monostable traveling wave fronts. Then, applying the techniques of weighted energy method and the comparison principle, we show that all solutions of the Cauchy problem for the discrete diffusive systems converge exponentially
to the traveling wave fronts when the initial perturbations around the wave fronts lie in a suitable weighted Sobolev space. Our main results can be extended  to more general discrete diffusive systems. We also apply them to  the discrete epidemic model with the Holling-II type and Richer type effects.
\end{abstract}
\vspace{0.25cm}
\vspace{0.25cm}
\newpage
\section{Introduction}
This paper is concerned with the wave propagation and its stability for a class of discrete diffusive systems. Such discrete systems arise in many
applications, e.g., the pulse propagation through myelinated
nerves \cite{Andersion-Sleeman}, the motion of domain walls in
semiconductor superlattices \cite{Carpio-Bonilla}, the sliding of charge
density waves \cite{Gruner}, and so on. Among these models, one can see the spatial discrete effects play important roles.
However, due to the special and poorly
understood phenomena occurring in these systems, the mathematical study of spatially
discrete models is more more difficult than that spatially
continuous models. Of particular phenomena is the pinning or propagation failure of wave fronts in spatially discrete equations. In past years, there were some significant progress on these subjects. We only illustrate some related works in the sequel.\medskip

In 1987, Keener \cite{Keener} studied the
propagation failure of wave fronts in coupled FitzHugh-Nagumo
systems of  discrete excitable cells
 \begin{align}\label{0.01}
 \left\{ \begin{aligned}
      \frac{dv_{1,j}(t)}{dt} =&d[v_{1,j+1}(t)-2v_{1,j}(t)+v_{1,j-1}(t)]+h(v_{1,j}(t),v_{2,j}(t)),\\
      \frac{dv_{2,j}(t)}{dt} =&g(v_{1,j}(t),v_{2,j}(t)),
  \end{aligned}
  \right.
\end{align}
where the subscript $j$ indicates the $j$th cell in a string of cells,$v_{1,j}$ represents the membrane potential of the cell and $v_{2,j}$ comprises additional variables (such as gating variables, chemical concentrations, etc.) necessary to the model. The constant $d$ means the coupling coefficient. Especially, if \eqref{0.01} is in the absent of the recovery, that is, $g\equiv0$ and $v_{2,j}$ is the constant independent of $j$, then \eqref{0.01} can reduce to a simple but
typical spatially discrete equation
\begin{align}\label{0.02}
      \frac{dv_{j}(t)}{dt} =&d[v_{j+1}(t)-2v_{j}(t)+v_{j-1}(t)]+f(v_{j}(t)).
      \end{align}
If the nonlinearity $f(\cdot)$ is a bistable function (e.g., $f(x)=x(x-a)(1-x)$), Bell and Cosner \cite{Bell-Cosner} obtained the threshold properties, that is, conditions forcing non-convergence to zero of solutions as time approaches infinity, and bounds on the speed of propagation of a ``wave of excitation". Later, Keener \cite{Keener} investigated the wave propagation and its failure for \eqref{0.02}. Subsequently, Zinner \cite{Zinner-91,Zinner-92} further considered the existence and stability of traveling wave fronts for \eqref{0.02}. Moreover, for the monostable nonlinearity $f(\cdot)$ (e.g., $f(x)=x(1-x)$), Zinner et al. \cite{Zinner-Harris-Hudson} established the existence of traveling wavefronts for the discrete Fisher's equation. Following the work of \cite{Zinner-Harris-Hudson}, there have been extensive studies on the propagation phenomenas of traveling wavefronts for more general monostable discrete equations, see  e.g., \cite{Chen-Fu-Guo,Chen-Guo02,Chen-Guo03,Fu-Guo-Shieh} and the references cited therein. \medskip

 Motivated by the system \eqref{0.01}, in this paper we are mainly concerned with the existence and stability of traveling wave fronts for the following two-component discrete diffusion system:
 \begin{align}\label{1.1}
 \left\{ \begin{aligned}
      \partial_t v_1(x,t) =&d_1\mathcal {D}[v_1](x,t)+h(v_1(x,t),v_2(x,t)),\\
      \partial_t v_2(x,t) =&d_2\mathcal {D}[v_2](x,t)+g(v_1(x,t),v_2(x,t)),
  \end{aligned}
  \right.
\end{align}
where $t>0 $, $x\in\mathbb{R}$, $d_i\geq0$, $h(u,v),g(u,v)\in C^2(\mathbb{R}^2,\mathbb{R})$ and
\begin{equation*}
\mathcal {D}[v_i](x,t):=v_i(x+1,t)-2v_i(x,t)+v_i(x-1,t),\,\,i=1,2.
\end{equation*}
System \eqref{1.1} can be considered as the continuum version of the lattice differential system
 \begin{align}\label{0.04}
 \left\{ \begin{aligned}
     {v_{1,j}^\prime(t)}=&d_1\mathcal {D}_j[v_{1}](t)+h(v_{1,j}(t),v_{2,j}(t)),\\
     {v_{2,j}^\prime(t)}=&d_2\mathcal {D}_j[v_{2}](t)+g(v_{1,j}(t),v_{2,j}(t)),
  \end{aligned}
  \right.\qquad\ \ \ \qquad
\end{align}
where $t>0 $, $j\in\mathbb{Z}$ and
$\mathcal {D}_j[v_i](t)=v_{i,j+1}(t)-2v_{i,j}(t)+v_{i,j-1}(t),\,\,i=1,2.$
Meanwhile, both systems \eqref{1.1} and \eqref{0.04} are spatial discrete versions
of the following reaction-diffusion system
\begin{align}\label{1.3}
 \left\{ \begin{aligned}
      \partial_t v_1(x,t) =&d_1\partial_{xx}v_1(x,t)+h(v_1(t),v_2(t)),\\
      \partial_t v_2(x,t) =&d_2\partial_{xx}v_2(x,t)+g(v_1(t),v_2(t)),
  \end{aligned}
  \right.\qquad\  \qquad
\end{align}
where $t>0 $, $x\in\mathbb{R}$. Systems \eqref{1.3} with special kinds of nonlinearities arise from many biological, chemical models, and so on (see \cite{Capasso-Wilson,Lewis-Schmitz,Murray}). For example, the system
\begin{align}\label{1.4}
 \left\{ \begin{aligned}
      \partial_t v_1(x,t) =&d_1\partial_{xx}v_1(x,t)-\alpha v_1(x,t)+h(v_2(x,t)),\\
      \partial_t v_2(x,t) =&d_2\partial_{xx}v_2(x,t)-\beta v_2(x,t)+g(v_1(x,t)),
  \end{aligned}
  \right.
\end{align}
with $\alpha,\beta>0$ describes the spread of an epidemic by oral-faecal transmission. Here $-\alpha v_1$ means  the natural death rate of the bacterial population; $-\beta v_2$ represents the natural diminishing rate of
the infective population due to the finite mean duration of the infectious population. The nonlinearity $h(v_2)\in C^2(\mathbb{R},\mathbb{R})$ is the contribution of the infective humans to the growth rate of the bacterial, while $g(v_1)\in C^2(\mathbb{R},\mathbb{R})$ is the infection rate of the human population.
For system \eqref{1.4},
Hsu and Yang \cite{Hsu-Yang} investigated the existence, uniqueness and asymptotic behavior of traveling waves for \eqref{1.4}. More recently, using the monotone iteration scheme via an explicit
construction of a pair of upper and lower solutions, the techniques of weighted energy method and comparison principle, Hsu et al. \cite{Hsu-Yang-Yu-17} extended \eqref{1.4} to more general delayed systems and obtained the existence and stability of traveling waves. For the lattice system \eqref{0.04}, Guo and Wu \cite{Guo-Wu09,Guo-Wu} recently investigated the existence of entire solutions and traveling wave fronts and its properties for the two-component spatially discrete competitive system
\begin{align}\label{0.03}
 \left\{ \begin{aligned}
    {v_{1,j}^\prime(t)}=&\mathcal {D}_j[v_{1}](t)+v_{1,j}(t)(1-v_{1,j}(t)-b_2v_{2,j}(t)),\\
    {v_{2,j}^\prime(t)}=&d\mathcal {D}_j[v_{2}](t)+rv_{2,j}(t)(1-v_{2,j}(t)-b_1v_{1,j}(t)),
  \end{aligned}
  \right.
\end{align}
for $d>0$. In addition, if we replace the terms $\mathcal {D}_j[v_{1}](t)$ and $\mathcal {D}_j[v_{2}](t)$ of \eqref{0.03} by ${D}[v_{1}](t)$ and ${D}[v_{2}](t)$ respectively, one can see that the profile equations of the new system are the same with those of \eqref{0.03} (cf. Section 2). Therefore, the new system also admits traveling wave fronts. \medskip

It is known that traveling wave solutions of biological models always correspond to the distribution of species and dynamics of phenomena. Therefore, it is significant to see whether the traveling wave solutions are stable or not.
Motivated by \cite{Guo-Wu,Hsu-Yang,Hsu-Yang-Yu-17}, we will investigate the existence and stability of traveling wave fronts of system \eqref{1.1}. \medskip

Recently, Hsu et al \cite{Hsu-Lin-Yang0}
considered the existence of traveling wave
solutions for the following lattice differential system:
\begin{equation}\label{1}
U'_{i,j}(t)=d_i{\mathcal D}_j[U_i](t)+f_i(U_{1,j},\cdots,U_{n,j}),\
\end{equation}
for $1\le i\le n$, $j\in\mathbb{Z}$ and $t\ge 0$, where $d_i>0$, $U_{i,j}(\cdot)\in C^2(\mathbb{R},\mathbb{R})$ and
$f_{i}(\cdot)\in C(\mathbb{R},\mathbb{R})$. Suppose the nonlinearities $f_i(\cdot)$ satisfy the following assumptions:
\begin{enumerate}
\item[(A1)]  System \eqref{1} has two homogeneous equilibria ${\bf 0}:=(0,\cdots,0)$ and ${\bf
E}:=(e_1,\cdots,e_n)$ with each $e_i>0$,
i.e. $
f_i({\bf 0})=f_{i}({\bf E})=0, \text{ for }1 \leq i \leq n. $
\item[(A2)] Assume that $\partial f_i(u)/\partial u_k\ge 0$ for all $u\in [\bf 0,\bf E]$ with $i\ne k$, $i,k=1,\cdots,n.$
Here the closed rectangle $[{\bf 0},{\bf E}]$ denotes the set
$\{u\in\mathbb{R}^n:{\bf 0}\le u\le{\bf E}\}.$
\item[(A3)] Each $f_i(\cdot)$ is Lipschitz continuous on
$[{\bf 0}, {\bf E}]$, and there exists a continuous function
 ${\bf r}=(r_1,\cdots,r_n):[0,1]\to [\bf 0,\bf E]$ with ${\bf r}(0)={\bf 0},{\bf r}(1)={\bf
 E}$ such that each $r_n$ is increasing and
$f_i\big({\bf r}(\varepsilon)\big)> 0, \text{ for
 }1\leq i\leq N \text{ and }\varepsilon \in (0,1).$
\end{enumerate}
Then the authors \cite{Hsu-Lin-Yang0} applied the truncated method to derive the following existence result of traveling wave
 solutions for system \eqref{1}.
\begin{thm}\label{existence} Assume that {\rm(A1)--(A3)} hold.
Suppose system \eqref{1}
has no other equilibrium in the closed rectangle $[{\bf 0}, {\bf E}]$,
then there exists $c^*>0$ such that if $c>c^*$ then system \eqref{1} has an increasing traveling wave solution
connecting {\bf 0} and {\bf E}.
\end{thm}

Since the profile equation of  \eqref{1.1} can be considered a special form as that of \eqref{1}, by Theorem \ref{existence}, we can directly obtain the existence of traveling wave fronts of system \eqref{1.1}. On the other hand, different to the assumption (A3), we can also derive the existence of traveling wave fronts for system \eqref{1.1} by using the monotone iteration method (see Theorem \ref{thm2.1}).\medskip

The stability of traveling wave fronts for reaction-diffusion equations with monostable nonlinearity has been extensively studied in past years, see \cite{Hsu-Yang-Yu-17,Ma-Zou,Mei-Lin-Lin-So1,Mei-Ou-Zhao,Wang-Li-Ruan1,Yang-Li-Wu1,Yu-Xu-Zhang}
and reference therein. For an example, Guo and Zimmer \cite{Guo-Zimmer} proved the global stability of traveling wave fronts for a spatially discrete equation by using a combination of the weighted energy method and the Green function technique. However, to the best of our knowledge, the stability of traveling wave solutions for multi-component discrete reaction diffusion systems is less reported. Recently, by comparison principles, Hsu and Lin \cite{Hsu-Lin} established a framework to study the stability of traveling wave solutions of the general system \eqref{1}. Unfortunately, due to different type of diffusion terms, their results can not be applied to system \eqref{1.1}. Motivated by these articles \cite{Guo-Zimmer,Hsu-Lin,Hsu-Yang-Yu-17,Mei-Ou-Zhao}, we will prove the stability of traveling wave fronts for the 2-component discrete system \eqref{1.1} by establishing the $L_{w_1}^1$, $L^1$ and $L^2$-energy estimates for the perturbation system (see Theorem \ref{thm2.2} and Section 4). Moreover, following the same proof arguments of the main theorem, we can extend the stability result to more general discrete diffusive system. \medskip

The rest of our paper is organized as follows. In Section 2, we introduce some necessary notations and present the main results
on the existence and stability of traveling wavefronts. Section 3 is devoted to analyzing the characteristic roots of the linearized equations. In Section 4, we prove the asymptotic stability of traveling wave fronts of \eqref{1.1} by using the the weighted energy method and comparison principle. Then, in Section 5, we extend the stability result of Theorem \ref{thm2.2} to the continuum version  of system \eqref{1}. Finally, we apply our main results to the discrete version of  epidemic model \eqref{1.4}.
\section{Main results}
\setcounter{equation}{0}

A solution $(v_1,v_2)$ or $(v_{1,j},v_{2,j})$ of system \eqref{1.1} or \eqref{0.04} is called a traveling wave solution
if there exist constant $c>0$ and smooth functions $\phi_1(\cdot)$ and $\phi_2(\cdot)$ such that
\begin{equation}\label{tr}
\mbox{or}\quad
\begin{array}{ll}
v_1(x,t)=\phi_1(x+ct)\ \mbox{and}\ v_2(x,t)=\phi_2(x+ct)\medskip\\
v_{1,j}(t)=\phi_1(j+ct)\ \mbox{and}\ v_{2,j}(t)=\phi_2(j+ct).
\end{array}
\end{equation}
Here $c$ means the wave speed and $\xi:=x+ct$ or $j+ct$ represents the moving coordinate.
Substituting the ans$\ddot{\rm a}$tzes of $\eqref{tr}$
into system \eqref{1.1} or  \eqref{0.04}, we can obtain the
same profile equations:
\begin{align}\label{2.1}
 \left\{ \begin{aligned}
      c\phi'_1(\xi) =&d_1\mathcal {D}[\phi_1](\xi)+h(\phi_1(\xi),\phi_2(\xi)),\\
      c\phi'_2(\xi) =&d_2\mathcal {D}[\phi_2](\xi)+g(\phi_1(\xi),\phi_2(\xi)),
  \end{aligned}
  \right.
\end{align}
where $\mathcal{D}[\phi](\xi):=\phi(\xi+1)-2\phi(\xi)+\phi(\xi-1)$.
Moreover, a traveling wave solution $(\phi_1,\phi_2)$ is called a
traveling wave front if each $\phi_i,~i=1,2$ is monotone. \medskip

To guarantee the existence of traveling wave solutions of \eqref{1.1}, throughout this article, we assume the nonlinearities $h(\cdot)$ and $g(\cdot)$ satisfy the following assumptions.
\begin{enumerate}
\item[{\rm(H1)}] System \eqref{1.1} have only two equilibria ${\bf 0}:=(0,0)$ and
${\bf K}:=(K_1,K_2)$ for some $K_1,K_2>0$ in the first quadrant, i.e.,
$h({\bf 0})=g({\bf 0})=0$ and $h({\bf K})=g({\bf K})=0$.
\item[{\rm(H2)}] $
  h_2(v):=\partial_2 { h}(v)\ge 0$ and
  $g_1(v):=\partial_1 { g}(v)\ge 0$,\
$\forall v\in{\bf I}:=[{\bf 0},{\bf K}].$
\item[{\rm(H3)}] $\alpha_i,\bar\alpha_i<0,$ for $i=1,2$,
$
\alpha_1\alpha_2<\beta_1\beta_2$ and $
\bar\alpha_1\bar\alpha_2>\bar\beta_1\bar\beta_2>0$, where
\begin{align*}
&\alpha_1:=\partial_{1} h({\bf 0}),\
\alpha_2:=\partial_{2} g({\bf 0}),\ \beta_1:=\partial_{2} h({\bf 0}),\ \beta_2:=\partial_{1} g({\bf 0}),\qquad\qquad\medskip \\
 &\bar\alpha_1:=\partial_{1} h({\bf K}),\
\bar\alpha_2:=\partial_{2} g({\bf K}),\
\bar\beta_1:=\partial_{2} h({\bf K}),\
\bar\beta_2:=\partial_{1} g({\bf K}).
\end{align*}
\end{enumerate}
Here we remark that two vectors $(u_1,\cdots,u_n)\le (v_1,\cdots,v_n)$ in $\mathbb{R}^n$ means $u_i\le v_i$ for $i=1,\cdots,n.$ An interval of
$\mathbb{R}^n$ is defined according to this order. \medskip

Based on the above assumptions, our first goal is to find solutions of \eqref{2.1} satisfying the following conditions:
\begin{equation}\label{2.2}
\lim_{ \xi\rightarrow-\infty}(\phi_1(\xi),\phi_2(\xi))={\bf 0}\,\,\,\mbox{and}\,\, \lim_{\xi\rightarrow\infty}(\phi_1(\xi),\phi_2(\xi))= {\bf K}.
\end{equation}
It's obvious that the (H1) and (H2) imply the assumptions (A1) and (A2) respectively. By the result of Theorem \ref{existence}, we immediately have the following existence result.
\begin{thm}\label{thm2.1}
Assume that $h(\cdot)$ and $g(\cdot)$ satisfy {\rm (H1), (H2)} and {\rm(A3)} or {\rm(H3)}. Then there exists a constant $c^*>0$ such that for any $c> c^*$, the system \eqref{1.1} admits a increasing traveling wave solution satisfying \eqref{2.2}.
\end{thm}
\begin{rem} {\rm
(1) The assumption (A3) in Theorem \ref{existence} could be verified for some specific systems, e.g., the Lotka-Volterra system, epidemic model and Nicholson's Blowflies reaction-diffusion equation (see \cite{Hsu-Lin-Yang0}). However, due to the assumption (A3), one can see from the proof of \cite{{Hsu-Lin-Yang0}} that system \eqref{1.1} may have increasing traveling wave solutions satisfying \eqref{2.2} when $c<c^*$. \medskip

(2) Since the nonlinearities of system \eqref{1.1} are not so general as \eqref{1}, to avoid the verification of (A3), we may replace (A3) by the assumption (H3) in Theorem \ref{thm2.1}.
Then, following the same ideas of our previous works \cite{Hsu-Yang,Hsu-Yang-Yu-17}, we can also obtain the same assertion of Theorem \ref{thm2.1}. In this situation, the constant $c^*$ is actually the threshold speed for the existence of increasing traveling wave solution satisfying \eqref{2.2}. More precisely, the assumption (H3) can help us to investigate the characteristic roots of the linearized equations for the profile equations \eqref{2.1} at the equilibria ${\bf 0}$ and ${\bf K}$. According to the local analysis of \eqref{2.1} at the  equilibria ${\bf 0}$ and ${\bf K}$ (see Section 3) and (H2), $c^*$ is actually the threshold speed such that the linearized  equation of \eqref{2.1} at ${\bf 0}$ has positive eigenvalues. By the eigenvalues, we can construct a pair of supersolution and subsolution for \eqref{2.1}, which are the same as those of \cite{Hsu-Yang}. Then, employing the monotone iteration scheme, system \eqref{1.1} admits traveling wave solutions satisfying \eqref{2.2}. Since the  proof arguments are the same as those of \cite{Hsu-Yang,Hsu-Yang-Yu-17}, we skip the details.}
\end{rem}
Next, we state the stability result of traveling wave fronts derived in Theorem \ref{2.2}. Before that, let us introduce the following notations.
\begin{enumerate}
\item[$\circ$] Let $I$ be an interval, especially $  I=\mathbb{R}$, then we denote $ L^2(I)$ by the space of the square integrable functions on $I$.
\item[$\circ$] The space $ H^{k}(I)\,\,(k\geq 0) $ means the Sobolev space of the $ L^2 $-functions $ f(x) $  defined on $I$ whose
derivatives $\frac {d^i}{dx^i}f (i=1,\cdots,k$) also belong to $L^2(I)$.
\item[$\circ$] Let's write $ L_{\omega}^2(I) $ and $W_\omega^{k,p}(I)$ by the weight $L^2$-space and weight Sobolev space with positive weighted function $\omega(x):\mathbb{R}\to\mathbb{R}$, respectively. For any $f\in  L_{\omega}^2(I) $ or $W_\omega^{k,p}(I)$, its norm is given (resp.) by
\begin{align*}
\|f\|_{L_w^2(I)}=\big(\int_{I} w(x)|f(x)|^2dx\big)^{1/2}, \ \mbox{or}\
\|f\|_{W_\omega^{k,p}(I)}=\big(\sum_{{i=0}}^{k}\int_{I}\omega(x)|\frac{d^i}{dx^i}f(x)|^p dx\big)^{1/p}.
\end{align*}
Furthermore, we set $H_w^k(I):=W_\omega^{k,2}(I)$.
\item[$\circ$] Letting $T>0$ and $\mathcal {B}$ be a Banach space, we denote by $C^0([0,T];\mathcal {B})$ the space of the $\mathcal {B}$-valued
continuous functions on $[0,T]$ and $L^2 ([0,T];\mathcal {B})$ as the space of the $\mathcal {B}$-valued $L^2$-function on $[0,T]$.
The corresponding spaces of the $\mathcal {B}$-valued functions on $[0,\infty)$ are defined similarly.
\end{enumerate}
%
%
In this paper, we select the weight function $\omega(\xi)$ as the form
\begin{align*}
\omega(\xi)=\max\{w_1(\xi),1\} =\begin{cases}
\omega_1(\xi), &\mbox {for}\,\,\,  \xi\leq\xi_0,\\
\quad\ \ 1, &\mbox {for}\,\,\,  \xi > \xi_0,\\
\end{cases}\,\,\, \mbox{with}\,\,\,  \omega_1(\xi):=e^{-\gamma(\xi-\xi_0)},
\end{align*}
where $\gamma$ and $\xi_0$ are positive constants which will be determined later. To obtain the stability result, we further assume $h(\cdot)$ and $g(\cdot)$ satisfy the following condition.
\begin{itemize}
\item [(H4)] $\partial_{ij}h(v)\le 0$ and $\partial_{ij}g(v)\le 0$, $\forall v\in\bf{I}$, $i,j=1,2$,
\begin{center}
$\bar\alpha_1+\bar\beta_2<0$, $\bar\alpha_2+\bar\beta_1<0$, $2\bar\alpha_1+\bar\beta_1+\bar\beta_2<0$ and $2\bar\alpha_2+\bar\beta_1+\bar\beta_2<0$.
\end{center}
\end{itemize}
Let $ (\phi_1(\xi),\phi_2(\xi)) $ be the a traveling wave solution
of \eqref{1.1} satisfying \eqref{2.2} with the wave speed $ c>c^* $.
Motivated by the work of \cite{Hsu-Yang-Yu-17,Mei-Ou-Zhao}, we will
adopt the weighted energy method to establish the $L^1$-weighted,
$L^1$- and $L^2$-energy estimates (see Section 4) for the
perturbations between solutions of \eqref{1.1} and $
(\phi_1(\xi),\phi_2(\xi)) $. Noting that (H4) plays an important
role in the weighted estimates. Then, by the comparison principle
and H\"{o}lder inequality, we can obtain the following stability
result.
\begin{thm}\label{thm2.2}Assume that {\rm (H1)}--{\rm(H4)} hold and the initial data of \eqref{1.1} satisfies
\begin{equation}\label{IC}
0\leq v_{i0}(x,0)\leq K_i,\ \forall x\in\mathbb R\ \mbox{and}\  v_{i0}(x,0)-\phi_i(x)\in C(L_\omega^1(\mathbb R)\cap H^1(\mathbb R))
\end{equation}
for $i=1.2$. Then the solution of \eqref{1.1} with initial data $ (v_{10}(x,0),v_{20}(x,0)) $ uniquely exists, which satisfies $0\leq v_{i}(x,t)\leq K_i,\ \forall(x,t)\in\mathbb R\times[0,+\infty)$ and
\begin{align*}
 v_{i}(x,t)-\phi_i(x+ct)\in C\big([0,+\infty), \,\,L_\omega^1(\mathbb R)\cap H^1(\mathbb R) \big),\ i=1,2.
\end{align*}
Moreover, there are positive constants $ \mu$ and $ C $ such that
\begin{align*}
\sup_{x\in\mathbb R}|v_i(x,t)-\phi_i(x+ct)| \leq Ce^{-\mu t},\ \forall t\ge 0,\ i=1,2.
\end{align*}
\end{thm}

Moreover, following the same proof arguments of Theorem \ref{thm2.2}, we can generalize the above stability result to the continuum version of system \eqref{1} (see Section 5).
\section{Local analysis for \eqref{2.1}}
\setcounter{equation}{0}
In this section, we will investigate the characteristic roots of the linearized equations for the profile equations \eqref{2.1} at
the equilibria ${\bf 0}$ and ${\bf K}$. Form \eqref{2.1} and the notations in (H3), one can see characteristic polynomials of \eqref{2.1} at ${\bf 0}$ and ${\bf K}$ have the form (resp.)
\begin{align*}
P(\lambda,c):=[d_1(e^{\lambda}+e^{-\lambda}-2) +\alpha_1-c\lambda][d_2(e^{\lambda}+e^{-\lambda}-2)+\alpha_2-c\lambda]-\beta_1\beta_2,\\
\bar{P}(\lambda,c):=[d_1(e^{\lambda}+e^{-\lambda}-2)+\bar\alpha_1-c\lambda][d_2(e^{\lambda}+e^{-\lambda}-2)+\bar\alpha_2-c\lambda]
-\bar\beta_1\bar\beta_2.
\end{align*}
Then the threshold speed $c^*$ in Theorem \ref{thm2.1} can be decided by the following lemma.
\begin{lem}\label{lem3.2}
Assume \rm{(H1)--(H3)} hold.
\begin{enumerate}
\item[{\rm(1)}] There exists a positive constant $ c^* $ such that if $ c> c^* $ then $P(c,\lambda)=0$ has two positive real roots $\lambda_1(c)<\lambda_2(c)<\lambda_m^+$, i.e.,
$P(\lambda_1,c)=P(\lambda_2,c)=0 $, and $P(\lambda,c)>0 $ for any $ \lambda\in (\lambda_1(c),\lambda_2(c))$. In addition, $\lim_{c\to c^*}\lambda_1(c)=\lim_{c\to c^*}\lambda_2(c)=\lambda^*>0$, i.e., $P(\lambda^*,c^*)=0$.
\item[{\rm(2)}] For any $c>0$, there exists a $\bar{\lambda}(c)>0$ such that $\bar{P}(\bar{\lambda},c)=0.$
Moreover, if $\varepsilon>0$ and small enough, we have $\bar{P}(\bar{\lambda}-\varepsilon)<0$.
\end{enumerate}
\end{lem}
\begin{proof} Since the proof is similar to \cite[Lemma 2.1]{Hsu-Yang}, we only sketch the proof of part (1) by the following four steps.\medskip

Step 1. Let's set
$
f_i(\lambda,c):=d_i(e^{\lambda}+e^{-\lambda}-2)+\alpha_i-c\lambda\,\text{ for }\,i=1,2.
$
It's easy to see that then there exist
$\lambda^-_i(c)<0<\lambda^+_i(c)$, $i=1,2$ such that $f_i(\lambda_i^\pm,c)=0$,
\begin{equation}\label{3.2}
  f_i(\lambda,c)<0, \mbox{ for } \lambda\in(\lambda_i^-,\lambda_i^+) \mbox{ and }
  f_i(\lambda,c)>0,\mbox{ for }\, \lambda\in\mathbb{R}\backslash[\lambda_i^-,\lambda_i^+],\ i=1,2.
\end{equation}

Step 2. Let's set $\lambda_M^\pm:=\max\{\lambda_1^\pm,\lambda_2^\pm\}$ and $\lambda_m^\pm:=\min\{\lambda_1^\pm,\lambda_2^\pm\}$. If $c$ is large enough, we have
$ P(1/\sqrt{c},c)>0\,\,\mbox{ and }\,\, 0<{1}/{\sqrt{c}}<\lambda_m^+<+\infty. $\medskip

Step 3. Let's define
\begin{align*}
 c^*:=\inf\{c_0>0: P(\lambda,c) \mbox{ has a root } \bar\lambda\in(0,\lambda_m^+) \mbox{ and } P^\prime(\bar\lambda,c)>0 \mbox{ for } c>c_0\}.
\end{align*}
If $ c>c^* $, then there exists some $ \lambda_0\in(0,\lambda_m^+) $
such that $ P(\lambda_0,c)>0 $. Since $ P(\pm\infty,c)=+\infty $,
$$
P(0,c)=\alpha_1\alpha_2-\beta_1\beta_2<0 \,\,\mbox{ and }\,\, P(\lambda_1^\pm,c)=P(\lambda_2^\pm,c)=-\beta_1\beta_2<0.
$$
Therefore, if $c$ is large enough, $ P(\lambda,c) $ has four roots in the following intervals
$$
(-\infty,0),\,\, (0,1/\sqrt{c}),\,\,(1/\sqrt{c},\lambda_m^+)\,\mbox{ and }\,(\lambda_M^+,+\infty).
$$

Step 4. Since
\begin{equation}\label{sketch}
P(\lambda,c)=f_1(\lambda,c)f_2(\lambda,c)-\beta_1\beta_2,
\end{equation}
following the proof arguments of  \cite[Lemma 2.1]{Hsu-Yang}, we can see that $ P(\lambda,c) $ has two positive real roots $\lambda_1(c)<\lambda_2(c)$ in $ (0,\lambda^+_m) $ which satisfy the assertions.
\end{proof}

Here we mention that the parameter $\gamma$ for the weighted function $\omega(\xi)$ will be chosen by $\gamma=\lambda_1(c)+\varepsilon$ (see Section 4), where $\varepsilon>0$ and small enough. Then it follows from (1) of Lemma \ref{3.2} that $P(\gamma)>0$.  Moreover, we recall the following lemma which plays an important role in obtaining the weighted energy estimate for the stability result.
\begin{lem}\label{lem4.1}{\rm (See \cite[Lemma 3.1]{Hsu-Yang}.)}
Let $A=(a_{ij})$ be a two by two matrix such that $a_{ii}<0, i=1,2$ and $a_{ij}>0$ for $i\neq j$. Then the system of the following inequalities
\begin{equation*}
\begin{array}{l}
  a_{11}x_1+a_{12}x_2<0\ (>0, resp.)\quad\mbox{and}\quad
  a_{21}x_1+a_{22}x_2<0\ (>0, resp.)
\end{array}
\end{equation*}
has a solution $(x_1,x_2)$ with $x_i>0,i=1,2$ if and only if $\det A>0\ (<0, resp.)$.
\end{lem}
\section{Stability of traveling wave fronts}
 \setcounter{equation}{0}
\qquad This section is devoted to prove the result of Theorem \ref{thm2.2}. To this end,
we first give some auxiliary statements about the global solutions of the Cauchy problem for \eqref{1.1} and the comparison principle. By the standard energy method and continuity extension method (see, \cite{Mei-So-Li-Shen}), we have the following result.
\begin{prop} \label{prop5.1}
Assume that {\rm (H1)--(H3)} hold, the initial data $ (v_{10}(x,0), v_{20}(x,0)) $ of system \eqref{1.1} satisfy the conditions of \eqref{IC}.
Then \eqref{1.1} admits a unique solution $(v_1(x,t),$ $v_2(x,t))$ such that $0\leq v_{i}(x,t)\leq K_i,\ \forall(x,t)\in\mathbb R\times[0,+\infty)$ and
\begin{align*}
 v_{i}(x,t)-\phi_i(x+ct)\in C\big([0,+\infty), \,\,L_\omega^1(\mathbb R)\cap H^1(\mathbb R) \big),\ i=1,2.
\end{align*}
\end{prop}

Similar to the proofs of \cite[Proposition 3]{Martin-Smith} and \cite[Lemma 3.2]{Thieme},
we easily obtain the following comparison principle.

\begin{prop}\label{lem5.1}{\rm (Comparison principle)}
Assume {\rm (H1)--(H3))}. Let $(v^\pm_{1}(x,t),v^\pm_{2}(x,t))$
be the solutions of system \eqref{1.1} with the initial data $(v^\pm_{10}(x,0),v^\pm_{20}(x,0))$, respectively. If
$
(v^{-}_{10}(x,0),v^{-}_{20}(x,0))\leq(v^{+}_{10}(x,0),v^{+}_{20}(x,0))
$
for all $x\in\mathbb R$, then it follows that
\begin{align*}
(v^{-}_{1}(x,t),v^{-}_{2}(x,t))\leq(v^{+}_{1}(x,t),v^{+}_{2}(x,t)),\
\forall (x,t)\in\mathbb{R}\times\mathbb{R}_+.
\end{align*}
\end{prop}

Hereinafter, we assume the initial data $(v_{10}(x,0), v_{20}(x,0))$ satisfies the conditions of \eqref{IC}, and set
\begin{equation*}
v_{i0}^-(x,0)\triangleq\min\{v_{i0}(x,0),\phi_i(x)\},\
v_{i0}^+(x,0)\triangleq\max\{v_{i0}(x,0),\phi_i(x)\},\ \mbox{for }i=1,2.
\end{equation*}
According to Proposition \eqref{prop5.1}, we denote $ (v_{1}^\pm(x,t),v_{2}^\pm(x,t)) $ by the nonnegative solutions
of system \eqref{1.1} with the initial data $ (v_{10}^\pm(x,0),v_{20}^\pm(x,0)) $. Then it follows from Proposition \ref{lem5.1} that
\begin{align}\label{5.1}
\left\{\begin{aligned}
&0\leq v_1^{-}(x,t) \leq v_1(x,t) ,\phi_1(x+ct) \leq v_1^{+}(x,t) \leq K_1,  \\
&0 \leq v_2^{-}(x,t) \leq v_2(x,t),\phi_2(x+ct) \leq v_2^{+}(x,t)\leq K_2,  \\
\end{aligned}\,\,\,\forall (x,t) \in \mathbb R \times \mathbb R_+.\right.
\end{align}
Therefore, the stability result of Theorem \ref{thm2.2} follows provided that
$ (v_1^\pm(x,t),v_2^\pm(x,t) )$ converges to $ (\phi_1(\xi),\phi_2(\xi) )$. For convenience, we denote
\begin{equation*}
V^\pm_{i}(\xi,t)\triangleq v_i^\pm(\xi-ct,t)-\phi_i(\xi),\ i=1,2,
\end{equation*}
with the corresponding initial data $ V_{10}^\pm(x,0)$ and $ V_{20}^\pm(x,0). $ Then our goal is to show that there exist positive constants $ C, \mu $, such that
\begin{align}\label{5.2}
\sup_{x\in\mathbb{R}}|V_{1}^\pm(\xi,t)|,\,\,\,\, \sup_{x\in\mathbb{R}}|V_{2}^\pm(\xi,t)|\leq Ce^{-\mu t},\,\,\mbox{for}\,\,t\geq0.
\end{align}
In the sequel, we only prove the assertion of \eqref{5.2} for $(V_1^+(\xi,t),V_2^+(\xi,t))$, since the the assertion for $(V_1^-(\xi,t),V_2^-(\xi,t))$
can be proved by the same way.

\subsection{$L_{\omega_1}^1$-energy and $L^1$-energy estimates}
For convenience,
we simplify the notations $ (V_{1}^+(\xi,t),V_{2}^+(\xi,t) )$ by $  (V_{1}(\xi,t),V_{2}(\xi,t) )$, and denote
$X(\xi,t):=(V_1(\xi,t),V_2(\xi,t))^T$ and
$\Phi(\xi):=(\phi_1(\xi),\phi_2(\xi))^T.$ By \eqref{1.1}, \eqref{2.1} and elementary computations, $V_1(\xi,t)$ and $V_2(\xi,t)$ satisfy
the system
\begin{align}\label{5.3}
\begin{cases}
&\ \ \ \partial_t V_1(\xi,t)+c\partial_\xi V_1(\xi,t)-d_1{\mathcal D}[V_1](\xi,t)=h(\Phi(\xi) +X(\xi,t))-h(\Phi(\xi) )\\
&=\nabla h(\Phi)X(\xi,t)+\displaystyle\frac{1}{2}[h_{11}(\tilde{\Phi})V_1^2+h_{22}(\tilde{\Phi})V_2^2+2h_{12}(\tilde{\Phi})V_1V_2],\medskip\\
&\ \ \ \partial_t V_2(\xi,t)+c\partial_\xi V_2(\xi,t)-d_2{\mathcal D}[V_2](\xi,t)=g(\Phi(\xi) +X(\xi,t) )-g(\Phi(\xi) )\\
&=\nabla g(\Phi)X(\xi,t)+\displaystyle\frac{1}{2}[g_{11}(\tilde{\Psi})V_1^2+g_{22}(\tilde{\Psi})V_2^2+2g_{12}(\tilde{\Psi})V_1V_2],
\end{cases}
\end{align}
with initial data
$
V_{10}(x,0)= v_{10}^+(x,0)-\phi_1(x)\ \mbox{and}\
V_{20}(x,0)= v_{20}^+(x,0)-\phi_2(x),
$
where
\begin{center}
$\Phi(\xi)\leq\tilde\Phi(\xi,t),\tilde\Psi(\xi,t)\leq\Phi(\xi) +X(\xi,t)$.
\end{center}
Obviously, $ V_{10}(x,0),V_{20}(x,0)\in C(L_\omega^1(\mathbb{R})\cap H^1(\mathbb{R})) $ and Proposition \ref{prop5.1}
implies that the solution $ V_1(\xi,t),V_2(\xi,t)\in C(L_\omega^1(\mathbb{R})\cap H^1(\mathbb{R})) $ for each $ t\in[0,+\infty) $. Furthermore, in order to establish the energy estimate, technically we need the sufficient regularity for the solution $V_1(\xi,t)$ and $V_2(\xi,t)$ of \eqref{5.3}.
To do this, the usual approach is applying the technique of  mollification. Let us mollify the initial data as
\begin{align*}
V_{i0}^{\varepsilon}(x,0)=J_{\varepsilon}\ast V_{i0}(x,0)=\int_{\mathbb{R}}J_{\varepsilon}(x-y)V_{i0}(y,0)dy\in W_\omega^{2,1}(\mathbb{R})\cap H^2(\mathbb{R}),\ i=1,2,
\end{align*}
where $ J_{\varepsilon}(\cdot) $ is the usual mollifier. Let $V_{1}^{\varepsilon}(\xi,t)$ and $V_{2}^{\varepsilon}(\xi,t)$ be the solution to {\eqref{5.3}} with the above mollified initial data.
We then have
\begin{align*}
V_{i}^{\varepsilon}(\xi,t)\in C([0,\infty),W_\omega^{2,1}(\mathbb{R})\cap H^2(\mathbb{R})),\ i=1,2.
\end{align*}
By taking the limit $\varepsilon\rightarrow 0$, we can obtain the corresponding energy estimate for original solution
$V_i(\xi,t)$ (cf. \cite[Lemma 3.1]{Mei-Ou-Zhao}). For the sake of simplicity, in the sequel we formally use $V_i(\xi,t)$ to establish the desired energy estimates.
\begin{lem}\label{lem5.2}
Assume that {\rm (H1)}--{\rm(H4)} hold. For any $c>c^*$ and $ \gamma=\lambda_1(c)+\varepsilon $, where $ \varepsilon>0 $ is small enough, there exist positive constants $ \mu $ and $ C $ such that
\begin{equation}\label{5.7}
\|V_1(\cdot,t)\|_{L^1_{\omega_1}(\mathbb{R})}+\|V_2(\cdot,t)\|_{L^1_{\omega_1}(\mathbb{R})}
+\int_0^t e^{\mu(s-t)}\big(\|V_1(\cdot,s)\|_{L^1_{\omega_1}}+V_2(\cdot,s)\|_{L^1_{\omega_1}}\big) ds
<Ce^{-\mu t}
\end{equation}
for each $t\ge 0$, where $ \omega_1(\xi)=e^{-\gamma(\xi-\xi_0)} $.
\end{lem}
\begin{proof} According to \rm{(A3)} and \eqref{5.3}, we have
\begin{align}\label{5.4}
\begin{cases}
&\partial_t V_1(\xi,t)+c\partial_\xi V_1(\xi,t)-d_1{\mathcal D}[V_1](\xi,t)-\nabla h(\Phi)X(\xi,t)\leq 0,\medskip\\
&\partial_t V_2(\xi,t)+c\partial_\xi V_2(\xi,t)-d_2{\mathcal D}[V_2](\xi,t)-\nabla g(\Phi)X(\xi,t)\leq 0.
\end{cases}
\end{align}
Multiplying the equations of \eqref{5.4} by $ e^{\mu t}\omega_1(\xi) $ for some $ \mu>0$  and integrating it over $ \mathbb{R}\times [0,t] $
with respect to $\xi$ and $t$,
since $V_i\in L_\omega^1(\mathbb R)\cap H^1(\mathbb R)\subseteq L_{\omega_1}^1(\mathbb R)\cap H^1(\mathbb R)$,
$\{e^{\mu t}\omega_1V_i\}|_{\xi=-\infty}^{\infty}=0\,\,(i=1,2)$, it follows that
\begin{align*}
0 \ge&\int_0^t\int_{-\infty}^\infty\big( \{e^{\mu s}\omega_1V_1(\xi,s)\}_s+\{ce^{\mu s}\omega_1V_1(\xi,s)\}_\xi
+e^{\mu s}\omega_1V_1(\xi,s)(-\mu+c\gamma-\alpha_1)\\
&\qquad\qquad\qquad\quad\qquad\qquad-e^{\mu s}d_1\omega_1{\mathcal D}[V_1](\xi,s)-e^{\mu s}\omega_1\beta_1V_2(\xi,s)\big)d\xi ds\\
= &e^{\mu t}\|V_1(\cdot,t)\|_{L^1_{\omega_1}(\mathbb{R}) }-\|V_{1}(\cdot,0)\|_{L^1_{\omega_1}(\mathbb{R})}
-\int_0^t\int_{-\infty}^\infty e^{\mu s}\omega_1\beta_1 V_2(\xi,s)d\xi ds\\
&+\int_0^t\int_{-\infty}^\infty e^{\mu s}\omega_1 V_1(\xi,s)[-\mu-f_1(\gamma,c)]d\xi ds.
\end{align*}
Note that $f_i(\lambda,c),\ i=1,2$ are defined in Lemma \ref
{lem3.2}. Hence, we have
\begin{align}\label{5.5}
e^{\mu t}\|V_1(\cdot,t)\|_{L^1_{\omega_1}(\mathbb{R}) }+&\int_0^t\int_{-\infty}^\infty e^{\mu s}\omega_1
V_1(\xi,s)[-\mu-f_1(\gamma,c)]d\xi ds \nonumber\\
-&\int_0^t\int_{-\infty}^\infty e^{\mu s}\omega_1\beta_1 V_2(\xi,s)d\xi ds\leq C_1,
\end{align}
for some constant $C_1>0$. Similarly, from the second equation of \eqref{5.4}, it yields
\begin{align}\label{5.6}
e^{\mu t}\|V_2(\cdot,t)\|_{L^1_{w_1}(\mathbb{R}) }+&\int_0^t\int_{-\infty}^\infty e^{\mu s}w_1
V_2(\xi,s)[-\mu-f_2(\gamma,c)]d\xi ds \nonumber\\
-&\int_0^t\int_{-\infty}^\infty e^{\mu s}w_1\beta_2 V_1(\xi,s)d\xi ds\leq C_2,
\end{align}
for some constant $C_2>0$. Since $\gamma=\lambda_1+\varepsilon$ with small $\varepsilon>0$, from the proof of Lemma \ref
{lem3.2} we know that $f_i(\gamma,c)<0$ for $i=1,2.$
Then it follows from Lemma \ref{lem4.1} that there are positive constant $p$ and $q$ satisfying the inequalities
\begin{align}\label{key}
&p f_1(\gamma,c)+q\beta_2<0\ \mbox{and}\
p\beta_1+q f_2(,\gamma,c)<0.
\end{align}
Multiplying \eqref{5.5} and \eqref{5.6} by the positive constants $p$ and $q$, respectively,
adding both inequalities, we can obtain
\begin{align*}
&p\|V_1(\cdot,t)\|_{L^1_{w_1}(\mathbb{R})}+q\|V_2(\cdot,t)\|_{L^1_{w_1}(\mathbb{R})}
-[p f_1(c,\gamma)+q \beta_2]\int_0^t\int_{-\infty}^\infty e^{-\mu (t-s)}w_1V_1(\xi,s)d\xi ds\nonumber\\
&-[p\beta_1+q f_2(c,\gamma)]\int_0^t\int_{-\infty}^\infty e^{-\mu (t-s)}w_1V_2(\xi,s)d\xi ds
\leq (pC_1+qC_2)e^{-\mu t}.
\end{align*}
Then, taking $\mu=0$, it follows that
\begin{align*}
&p\|V_1(\cdot,t)\|_{L^1_{w_1}(\mathbb{R}) }+q\|V_2(\cdot,t)\|_{L^1_{w_1}(\mathbb{R}) }
-[p f_1(c,\gamma)+q \beta_2]\int_0^t \|V_1(\cdot,s)\|_{L^1_{w_1}} ds\\
&-[q f_2(c,\gamma)+p\beta_1]\int_0^t \|V_2(\cdot,s)\|_{L^1_{w_1}} ds \leq pC_1|_{\mu=0}+qC_2|_{\mu=0}.
\end{align*}
By taking $\mu>0$ and small enough, it follows that
\begin{align*}
-p f_1(\gamma,c)-q\beta_2-p\mu>0  ~and~ -q f_2(\gamma,c)-p\beta_1-q\mu>0.
\end{align*}
Then we obtain the key energy estimate \eqref{5.7}.
This completes the proof.
\end{proof}

Using the $L^1_{\omega_1}$-estimate of Lemma \ref{lem5.2}, we further have the following $L^1$-estimate.
\begin{lem}\label{lem5.3}
Assume that {\rm (H1)}--{\rm(H4)} hold. Then, for any $c>c^*$,  there exists positive constants $ \mu ,\ \xi_0$ and $ C $ such that
\begin{align}\label{L1}
e^{\mu t}(\|V_1(\cdot,t)\|_{L^1 (\mathbb{R})}+\|V_2(\cdot,t)\|_{L^1 (\mathbb{R})})\leq C,\ \forall t\geq 0.
\end{align}
\end{lem}
\begin{proof}
Multiplying the equations \eqref{5.4} by $ e^{\mu t} $ and integrating it over $ \mathbb{R}\times[0,t] $ with respect to $\xi$ and $t$,
since $V_i\in L_\omega^1(\mathbb R)\cap H^1(\mathbb R)$,
$\{c e^{\mu t}V_i\}|_{\xi=-\infty}^{\infty}=0\,\,(i=1,2)$,
we can obtain
\begin{align}\label{5.8}
0 &\ge\int_0^t\int_{-\infty}^\infty\big( \{e^{\mu s}V_1(\xi,s)\}_s-\mu e^{\mu s}V_1(\xi,s)+\{ce^{\mu s}V_1(\xi,s)\}_\xi
-e^{\mu s}[d_1{\mathcal D}[V_1](\xi,s)\nonumber\\
&\ \qquad\qquad\quad +h_1(\Phi(\xi))V_1(\xi,s)+h_2(\Phi(\xi))V_2(\xi,s)]\big)d\xi ds\nonumber\\
&= e^{\mu t}\|V_1(\cdot,t)\|_{L^1(\mathbb{R}) }-\|V_{1}(\cdot,0)\|_{L^1(\mathbb{R})} \nonumber\\
&\qquad\qquad\qquad\qquad\  +\int_0^t\int_{-\infty}^\infty e^{\mu s}\big(\mathbf{F}_1(\xi)V_1(\xi,s)+\mathbf{F}_2(\xi)V_2(\xi,s)]\big) d\xi ds,
\end{align}
where $\mathbf{F}_1(\xi):=-\mu-h_1(\Phi(\xi))$ and $\mathbf{F}_2(\xi):=-h_2(\Phi(\xi))$.
Since $\omega_1(\xi)\geq1$ for $\xi\leq\xi_0$, by Lemma \ref{lem5.2}, we can obtain
\begin{align}\label{5.9}
&\big|\int_{0}^{t}\int_{-\infty}^{\xi_0}e^{\mu s}\big(\mathbf{F}_1(\xi)V_1(\xi,s)+\mathbf{F}_2(\xi)V_2(\xi,s)\big)d\xi ds\big|\nonumber\\
\le&C_3\int_{0}^{t}e^{\mu s}\big(\|V_1(\cdot,s)\|_{L_{\omega_1}^1(-\infty,\xi_0)}+\|V_2(\cdot,s)\|_{L_{\omega_1}^1(-\infty,\xi_0)}\big)ds\le {C_4},
\end{align}
for some positive constants $C_3$ and $C_4$. Then it follows from \eqref{5.8} and \eqref{5.9} that
\begin{align}\label{5.10}
&e^{\mu t}\|V_1(\cdot,t)\|_{L^1(\mathbb{R}) }+\int_0^t\int_{\xi_0}^\infty e^{\mu s}\big(\mathbf{F}_1(\xi)V_1(\xi,s)+\mathbf{F}_2(\xi)V_2(\xi,s)\big) d\xi ds\le C_5,
\end{align}
for some constant $C_5>0$. Similarly, there exists a constant $C_6>0$ such that
\begin{align}\label{5.11}
&e^{\mu t}\|V_2(\cdot,t)\|_{L^1(\mathbb{R}) }+\int_0^t\int_{\xi_0}^\infty e^{\mu s}\big(\mathbf{G}_1(\xi)V_1(\xi,s)+\mathbf{G}_2(\xi)V_2(\xi,s)\big) d\xi ds\le C_6,
\end{align}
where $\mathbf{G}_1(\xi):=-g_1(\Phi(\xi))$ and $\mathbf{G}_2(\xi):=-\mu-g_2(\Phi(\xi))$.
Summing \eqref{5.10} and \eqref{5.11}, there exists a constant $C>0$ such that
\begin{align}\label{5.111}
\|V_1(\cdot,t)\|_{L^1(\mathbb{R}) }+\|V_2(\cdot,t)\|_{L^1 (\mathbb{R})}
+\int_0^t\int_{\xi_0}^{\infty} &e^{-\mu (t-s)}\big[\big(\mathbf{F}_1(\xi)
+\mathbf{G}_1(\xi)\big)U(\xi,s)+\nonumber\\
&\big(\mathbf{F}_2(\xi)
+\mathbf{G}_2(\xi)\big)V(\xi,s)\big]d\xi ds\le Ce^{-\mu t}.
\end{align}
By the assumption (H4), we have
\begin{align*}
\lim_{\xi\to\infty}\big(\mathbf{F}_1(\xi)+\mathbf{G}_1(\xi)\big)=&-h_1({\bf K})-g_1({\bf K})=-\bar{\alpha}_1-\bar{\beta}_2>0,\\
\lim_{\xi\to\infty}\big(\mathbf{F}_2(\xi)+\mathbf{G}_2(\xi)\big)=&-h_2({\bf K})-g_2({\bf K})=-\bar{\beta}_1-\bar{\alpha}_2>0.
\end{align*}
Therefore, choosing $\xi_0$ large enough and $\mu$ small enough, we have
$\mathbf{F}_1(\xi)+\mathbf{G}_1(\xi),\mathbf{F}_2(\xi)+\mathbf{G}_2(\xi)>0$ for $\xi\geq\xi_0$.
Hence, the assertion of this lemma follows from \eqref{5.111}.
\end{proof}

\subsection{$L^2$-energy estimate}
Based on the $L_{\omega_1}^1$-estimate of Lemma \ref{lem5.2}, we further have the following $L^2$-estimate.
\begin{lem}\label{lem5.4}
Assume that {\rm (H1)}--{\rm(H4)} hold. For any $c>c^*$, there exist positive constants  $\xi_0$
{\rm(}for the weight function $w(\xi)${\rm)} and $C$ such that
\begin{align}\label{L21}
\|V_1(\cdot,t)\|^2_{L^2(\mathbb{R})}+\|V_2(\cdot,t)\|^2_{L^2(\mathbb{R})}\le C,\ \forall t\ge0.
\end{align}
\end{lem}
\begin{proof}
Let's multiply the equations \eqref{5.4} by $ V_1(\xi,t) $  and integrating it over $ \mathbb{R}\times[0,t] $ with respect to $\xi$ and $t$. Since $V_i\in L_\omega^1(\mathbb R)\cap H^1(\mathbb R)$,
$\{\frac{1}{2}c V_i^2\}|_{-\infty}^{\infty}=0\,\,(i=1,2)$. Thus, we can obtain that
\begin{align*}
0 &\ge\int_0^t\int_{-\infty}^\infty\big( \{\frac{1}{2}V_1^2(\xi,s)\}_s+\{\frac{1}{2}cV_1^2(\xi,s)\}_\xi -V_1(\xi,s)[d_1{\mathcal D}[V_1](\xi,s)+\nonumber\\
&\qquad\qquad\quad\ h_1(\Phi(\xi))V_1(\xi,s)+h_2(\Phi(\xi))V_2(\xi,s)]\big)d\xi ds\nonumber\\
&=\frac{1}{2}\big(\|V_1(\cdot,t)\|_{L^2(\mathbb{R})}-\|V_{1}(\cdot,0)\|_{L^2(\mathbb{R})}\big)\nonumber\\
&\qquad+\int_0^t\int_{-\infty}^\infty\big(-h_1(\Phi(\xi))V_1^2(\xi,s)-h_2(\Phi(\xi))V_1((\xi,s)V_2(\xi,s)\big)d\xi ds \nonumber\\
&\ge\frac{1}{2}\big(\|V_1(\cdot,t)\|_{L^2(\mathbb{R})}-\|V_{1}(\cdot,0)\|_{L^2(\mathbb{R})}\big)+\int_0^t\int_{-\infty}^\infty\big(\mathbf{Q}_1(\xi)V_1^2(\xi,s)+\mathbf{Q}_2(\xi)V_2^2(\xi,s)\big)d\xi ds,
\end{align*}
where $\mathbf{Q}_1(\xi):=-h_1(\Phi(\xi))-\frac{1}{2}h_2(\Phi(\xi))$ and $\mathbf{Q}_2(\xi):=-\frac{1}{2}h_2(\Phi(\xi))$.
Since $\omega_1(\xi)\geq1$ for $\xi\leq\xi_0$ and $0\leq V_1(\xi,t)\leq K_1$, Lemma \ref{lem5.3} implies that
\begin{align}\label{5.13}
\int_{-\infty}^{\xi_0}V_1^2(\xi,t)d\xi\leq K_1\int_{-\infty}^{\xi_0}\omega(\xi)V_1(\xi,t)d\xi
\leq K_1\|V_1(\cdot,t)\|_{L_{\omega_1}^1(\mathbb{R})}\leq K_1 C e^{-\mu t}.
\end{align}
Similarly, we can obtain
\begin{equation}\label{5.14}
\int_{-\infty}^{\xi_0}V_2^2(\xi,t)d\xi\leq K_2 C e^{-\mu t}.
\end{equation}
Then it follows that
\begin{align}\label{5.15}
&\big|\int_{0}^{t}\int_{-\infty}^{\xi_0}e^{\mu s}\big(\mathbf{Q}_1(\xi)V_1^2(\xi,s)+\mathbf{Q}_2(\xi)V_2^2(\xi,s)\big)d\xi ds\big|\nonumber\\
\le&C_7\int_{0}^{t}\int_{-\infty}^{\xi_0}V_1^2(\xi,s)+V_2^2(\xi,s)d\xi ds\leq C_8,
\end{align}
for some positive constants $C_7$ and $C_8$. Moreover, we have
\begin{align}\label{5.16}
\|V_1(\cdot,t)\|_{L^2(\mathbb{R})}^2+2\int_0^t\int_{\xi_0}^\infty \big(\mathbf{Q}_1(\xi)V_1^2(\xi,s)+\mathbf{Q}_2(\xi)V_2^2(\xi,s)\big) d\xi ds\leq C_9,
\end{align}
for some positive constants $C_9$. Similarly, there exists $C_{10}>0$ such that
\begin{align}\label{5.17}
\|V_2(\cdot,t)\|_{L^2(\mathbb{R})}^2+2\int_0^t\int_{\xi_0}^\infty \big(\mathbf{P}_1(\xi)V_1^2(\xi,s)+\mathbf{P}_2(\xi)V_2^2(\xi,s)\big) d\xi ds\leq C_{10},
\end{align}
where $\mathbf{P}_1(\xi):=-\frac{1}{2}g_1(\Phi(\xi))$ and $\mathbf{P}_2(\xi):=-\frac{1}{2}g_1(\Phi(\xi))-g_2(\Phi(\xi))$.
Summing \eqref{5.16} and  \eqref{5.17}, it follows
\begin{align}\label{5.18}
\|V_1(\cdot,t)\|_{L^2(\mathbb{R})}^2+\|V_2(\cdot,t)\|_{L^2(\mathbb{R})}^2+2\int_0^t\int_{\xi_0}^\infty& \big(\big(\mathbf{P}_1(\xi)+\mathbf{Q}_1(\xi)\big)V_1^2(\xi,s)+ \nonumber\\
&\big(\mathbf{P}_2(\xi)+\mathbf{Q}_2(\xi)\big)V_2^2(\xi,s)\big) d\xi ds\leq C_{11}.
\end{align}
for some $C_{11}>0$. By the assumption {\rm(H4)}, we have
\begin{align*}
\lim_{\xi\to\infty}\big(\mathbf{Q}_1(\xi)+\mathbf{P}_1(\xi)\big)&=-h_1({\bf K})-\frac{1}{2}h_2({\bf K})-\frac{1}{2}g_1({\bf K})=-\frac{1}{2}\big(2\bar{\alpha}_1+\bar{\beta}_2+\bar{\beta}_1\big)>0,\nonumber\\
\lim_{\xi\to\infty}\big(\mathbf{Q}_2(\xi)+\mathbf{P}_2(\xi)\big)&=-\frac{1}{2}h_2({\bf K})-\frac{1}{2}g_1({\bf K})-g_2({\bf K}) =-\frac{1}{2}\big(\bar{\beta}_1+\bar{\beta}_2+2\bar{\alpha}_2\big)>0.
\end{align*}
Therefore, choosing $\xi_0$ large enough, we have $\mathbf{Q}_1(\xi)+\mathbf{P}_1(\xi),~\mathbf{Q}_2(\xi)+\mathbf{P}_2(\xi)>0$ for $\xi\geq\xi_0$.
Then the estimate \eqref{L21} follows from \eqref{5.18}.
The proof is complete.
\end{proof}
By the same procedure, we can also obtain the $L^2$-estimate for the derivatives $\partial_{\xi}V_1(\cdot,t)$ and $\partial_{\xi}V_2(\cdot,t)$.
Indeed, differentiating the system \eqref{5.4} with respect to $\xi$, we can obtain
\begin{equation*}
\left\{
\begin{array}{ll}
\partial_{t\xi} V_1(\xi,t)+c\partial_{\xi\xi} V_1(\xi,t)-d_1{\mathcal D}[\partial_{\xi}V_1](\xi,t)
-\nabla h(X(\xi,t)+\Phi(\xi))\partial X_{\xi}(\xi,t)\leq 0,\medskip\\
\partial_{t\xi} V_2(\xi,t)+c\partial_{\xi\xi} V_2(\xi,t)-d_2{\mathcal D}[\partial_{\xi}V_2](\xi,t)
-\nabla g(X(\xi,t)+\Phi(\xi))\partial X_{\xi}(\xi,t)\leq 0.
\end{array}
\right.
\end{equation*}
Similar to the proof of Lemma \ref{lem5.4}, we can obtain the following $L^2$-estimate.
\begin{lem}\label{lem5.5}
Assume that {\rm (H1)}--{\rm(H4)} hold. Then, for any $c>c^*$, there exist positive constants $ \xi_0 $ and $ C $ such that
\begin{align*}
\|\partial_{\xi}V_1(\cdot,t)\|^2_{L^2(\mathbb{R})}+\|\partial_{\xi}V_2(\cdot,t)\|^2_{L^2(\mathbb{R})}\le C,\ \forall t\ge 0.
\end{align*}
\end{lem}

\subsection{Proof of Theorem 2.2}

It is easy to see that
\begin{align}\label{5.20}
\|V_1(\cdot,t)\|_{L^2(\mathbb{R})}^2\leq \sup_{\xi\in\mathbb{R}}|V_1(\xi,t)|\int_{\mathbb{R}}V_1(\cdot,t)d\xi  =\|V_1(\cdot,t)\|_{L^\infty(\mathbb{R})}\cdot \|V_1(\cdot,t)\|_{L^1(\mathbb{R})},
\end{align}
for any $t\ge0$. By the H\"{o}lder inequality, we have
\begin{align*}
V_1^2(\cdot,t)&=2\int_{-\infty}^\xi \partial_{\xi}V_1(\cdot,t) V_1(\cdot,t)d\xi \leq 2\big(\int_{-\infty}^\xi |\partial_{\xi}V_1(\cdot,t)|^2 d\xi\big)^{\frac{1}{2}}
\big(\int_{-\infty}^\xi |V_1(\cdot,t)|^2 d\xi\big)^{\frac{1}{2}}\\
&\leq 2\|V_1(\cdot,t)\|_{L^2(-\infty,\xi)}\cdot \|\partial_{\xi}V_1(\cdot,t)\|_{L^2(-\infty,\xi)},
\end{align*}
for any $ \xi\in\mathbb{R} $ and $ t\geq 0 $. Then it follows that
\begin{align}\label{5.21}
&\|V_1(\cdot,t)\|_{L^\infty(\mathbb{R})}^2\le2 \|V_1(\cdot,t)\|_{L^2(\mathbb{R})}\cdot \|\partial_{\xi}V_1(\cdot,t)\|_{L^2(\mathbb{R})}.
\end{align}
Combing \eqref{5.20} and \eqref{5.21}, we have
$$
\|V_1(\cdot,t)\|_{L^\infty(\mathbb{R})}\leq2^{\frac{2}{3}}\|V_1(\cdot,t)\|_{L^1(\mathbb{R})}^{\frac{1}{3}}
\cdot \|\partial_{\xi}V_1(\cdot,t)\|_{L^2(\mathbb{R})}^{\frac{2}{3}},\ \forall  t\ge 0.
$$
According to Lemmas \ref{lem5.3} and \ref{lem5.5}, there exist positive constants $ \mu_1 $ and $ M_1 $, such that
$$
\|V_1(\cdot,t)\|_{L^\infty(\mathbb{R})}\leq M_1e^{-\frac{1}{3}\mu_1 t},\ \forall t\ge 0.
$$
Similarly, there exist $ \mu_2>0 $ and $ M_2>0 $, such that
$$
\|V_2(\cdot,t)\|_{L^\infty(\mathbb{R})}\leq M_2e^{-\frac{1}{3}\mu_2 t},\ \forall t\ge 0.
$$
Thus, we have
\begin{align*}
\sup_{x\in\mathbb{R}}|v_i^+(x,t)-\phi_i(x+ct)|\leq Ce^{-\mu t},
\forall t\ge 0,\ i=1,2.
\end{align*}
Similarly, we can verify that for any $ c>c^* $, it holds
\begin{align*}
\sup_{x\in\mathbb{R}}|v_i^-(x,t)-\phi_i(x+ct)|\leq Ce^{-\mu t},\ \forall t\ge 0, \ i=1,2.
\end{align*}
Since
$
0\le v_{i}^-(\xi,t)\le v_{i}(\xi,t),\,\,\phi_i(x+ct)\le v_{i}^+(\xi,t)\le K_i,\ i=1,2,
$
the squeezing argument implies that
\begin{align*}
\sup_{x\in\mathbb{R}}|v_i(x,t)-\phi_i(x+ct)|\leq Ce^{-\mu t},\ \forall t\geq0,\ i=1,2.
\end{align*}
This completes the proof of Theorem \ref{thm2.2}.  \qed

\section{Extension to general discrete diffusive system}
In this section, we will generalize the result of Theorem \ref{thm2.2} to the following discrete diffusive system
\begin{align}\label{nsystem}
\partial_tU_{i}(x,t)&=d_i{\mathcal D}[U_i](x,t)+f_i(U_{1}(x,t),\cdots,U_{n}(x,t)),\ \mbox{for }i=1,\cdots,n.
\end{align}
We assume that the conditions (A1)--(A3) hold for \eqref{nsystem}. Since the profile equations of \eqref{nsystem} are the same with those of \eqref{1}, we can have the existence result of traveling wave solutions as the statement of Theorem \ref{existence}. According to Section 4, we know that the results of Lemmas \ref{lem3.2}--\ref{lem4.1} are significant in proving the estimations of Lemmas \ref{lem5.2}--\ref{lem5.5}. Therefore, to obtain the stability of traveling wave solutions of \eqref{nsystem}, we have to generalize the results of Lemma \ref{lem3.2}--\ref{lem4.1}. \medskip

Recently, Hsu and Yang \cite{CHY} generalized the statement of Lemma \ref{lem4.1} to more general cases. Before to cite their results, we first introduce the following notations.\medskip

Let $A=(a_{i,j})\in M_{n\times n}(\mathbb{R})$, for any $1\le k\le n$, we define the submatrix $A_{p_1,p_2,\cdots,p_k}:=(a_{p_i, p_j})_{1\le i,j\le k}$ for any subset $\{p_1,\cdots,p_k\}\subseteq\{1,\cdots,n\}$.
Following these notations, Hsu and Yang \cite{CHY} recently proved the following result.
\begin{lem}\label{lemnsystem}
Let $A=(a_{i,j})\in M_{n\times n}(\mathbb{R})$ with $a_{i,j}\ge 0$ for
$i\neq j$ and $a_{i,i}\le 0$ for $1\le i\le n$. Then the system of
inequalities $Ax< 0$ has a solution $x=(x_1,\cdots,x_n)^T$ with
each $x_i> 0$ if and only if for each $1\le k\le n$,
\begin{equation}
(-1)^{k-1}\det(A_{1,2,\cdots,k,j})>0,\ \mbox{for }k+1\le j\le n. \label{equ:det}
\end{equation}
\end{lem}
Note that Lemma \ref{lem4.1} is a special case of the above lemma with $n=2$.
Now we consider the profile equation for system \eqref{nsystem}, that is
\begin{equation}\label{npf}
c\phi_i^\prime(\xi)=d_i{\mathcal D}[\phi_i](\xi)+f_i(\phi_1(\xi),\cdots,\phi_n(\xi)),\ i=1,\cdots,n,
\end{equation}
here $\xi=x+ct$ and $U_i(x,t)=\phi_i(x+ct)$ for $i=1,\cdots,n$. Let's set
\begin{center}
$\alpha_{i,j}:=\partial f_i({\bf 0})/\partial u_j$, $\delta_{i,j}:=\alpha_{i,j}$ if $i\ne j$, and
$\delta_{i,i}=\delta_{i,i}(\lambda,c):=d_i(e^{\lambda}+e^{-\lambda}-2)-c\lambda+\alpha_{i,i}$,
\end{center}
for $i,j=1,\cdots,n$.
 Then the characteristic polynomial of \eqref{npf} at {\bf 0} has the form
\begin{align*}
P(\lambda,c)=&\det J(\lambda,c):=\det[\delta_{i,j}].
\end{align*}
\begin{lem}\label{speed}
There exist $c_*>0$ and $0<\lambda_1(c)<\lambda_2(c)$ such that $\delta_{i,i}<0$ for all $\lambda\in(\lambda_1(c),\lambda_2(c))$, $i=1,\cdots, n$, provided that $c>c_*$.
\end{lem}
\begin{proof}Let's define
\begin{equation}
c_*:=\min_{\lambda>0}\frac{d_M(e^{\lambda}+e^{-\lambda}-2)+\alpha_M}{\lambda},\  i=1,\cdots,n,
\end{equation}
where $\alpha_M:=\max\{\alpha_{1,1},\cdots,\alpha_{n,n}\}$ and $d_M:=\{d_1,\cdots,d_n\}$.
If $c>c_*$ there exist $0<\lambda_1(c)<\lambda_2(c)$ such that
\begin{equation*}
 c\lambda>d_M(e^{\lambda}+e^{-\lambda}-2)+\alpha_M,\ \mbox{for }\lambda\in(\lambda_1(c),\lambda_2(c)).
 \end{equation*}
 Then, for any $i=1,\cdots,n$, we have
\begin{equation*}
 \delta_{i,i}<d_M(e^{\lambda}+e^{-\lambda}-2)-c\lambda+\alpha_M<0,\ \mbox{for }\lambda\in(\lambda_1(c),\lambda_2(c)).
 \end{equation*}
 The proof is complete.
\end{proof}
As a consequence of Lemmas \ref{lemnsystem} and \ref{speed}, we immediately have the following result.
\begin{lem}\label{root}
Assume all $\delta_{i,j}\ge 0$ with $i\ne j$, $c>c_*$ and $\lambda\in(\lambda_1(c),\lambda_2(c))$. Then the exists a vector $v=(v_1,\cdots,v_n)^T\in\mathbb{R}^n$ with all $v_i>0$ such that
$J(\lambda,c)v<0$ if and only if for each $1\le k\le n$,
\begin{equation}\label{iff}
(-1)^{k-1}\det(J_{1,2,\cdots,k,j})(\lambda,c)>0,\ \mbox{for }k+1\le j\le n.
\end{equation}
\end{lem}

Furthermore, we generalize (H4) by the following assumptions
\begin{enumerate}
\item[{\rm(A4)}] ${\partial^2_{u_j u_k}f_i(u)}\le 0$, $\forall u\in\bf{I}$, $i,j,k=1,\cdots,n;$ and
\begin{align}
\sum_{i=1}^{n}\bar\alpha_{i,j}<0\ \mbox{and}\ \ 2\bar{\alpha}_{j,j}+\sum_{i\ne k,i=1}^{n}\bar\alpha_{i,k}<0,\  j=1,\cdots,n,
\end{align}
where $\bar{\alpha}_{i,j}:=\partial f_i({\bf K})/\partial u_j$, $i,j,=1,\cdots,n.$
\end{enumerate}
Note that (H4) is a special case of (A4) with $n=2$. The condition \eqref{iff} with $n=2$ is equivalent to Lemma \ref{lem4.1}.
According to the above lemmas, we assume (A1)--(A4), \eqref{iff} hold and $c>\max\{c^*,c_*\}$. Then we can also obtain the estimations of Lemmas \ref{lem5.2}--\ref{lem5.5}. More precisely,  similar to the previous notations,
let's write $ (V_{1}^+(\xi,t),\cdots,V_{n}^+(\xi,t) )$ by $  (V_{1}(\xi,t),\cdots,V_{n}(\xi,t) )$, and denote
$X(\xi,t):=(V_1(\xi,t),\cdots,V_n(\xi,t))^T$ and
$\Phi(\xi):=(\phi_1(\xi),\cdots,\phi_n(\xi))^T.$ By elementary computations, \eqref{5.3} is generalized to the following
the system
\begin{align}\label{5.35}
\partial_t V_i(\xi,t)+c\partial_\xi V_i(\xi,t)-d_i{\mathcal D}[V_i](\xi,t)&=f_i(\Phi(\xi) +X(\xi,t))-f_i(\Phi(\xi) )\nonumber\\
&=\nabla f_i(\Phi)X(\xi,t)+\frac{1}{2}\sum\limits_{j,k=1,2}\displaystyle\frac{\partial^2 f_i(\tilde{\Phi}_i)}{\partial u_ju_k}V_jV_k,
\end{align}
for $i=1,\cdots,n$, where $\Phi(\xi)\leq\tilde\Phi_i(\xi,t)\leq\Phi(\xi) +X(\xi,t)$.
Let's replace the parameters $\gamma$ and $(p,q)$ in the proof of Lemma \ref{lem5.2} by $\lambda$ and $v$ of Lemma \ref{root}, respectively. Then \eqref{key} yields to
\begin{equation*}
J(\lambda,c)v=[\delta_{i,j}]v<0.
\end{equation*}
Then the proof of Lemma \ref{lem5.2} also true.
In addition, it is easy to see that the proofs of Lemmas \ref{lem5.3}--\ref{lem5.5} also hold under the the assumption (A4). Hence, we have the following stability result.
\begin{thm}\label{thm1}
Assume {\rm(A1)--(A4)}, \eqref{iff} hold and $c>\max\{c^*,c_*\}$. System \eqref{nsystem} admits a traveling wave solution connecting ${\bf 0}$ and ${\bf K}$, which is exponential stable in the same sense as that of {\rm Theorem \ref{thm2.2}}.
\end{thm}
\section{Applications}
\setcounter{equation}{0}
\qquad In this section, we will apply the main theorem to the  discrete version of epidemic model \eqref{1.4}, that is
\begin{align}\label{6.1}
 \left\{ \begin{aligned}
      \partial_t v_1(x,t) =&d_1\mathcal {D}[v_1](x,t)-a_1v_1(x,t)+\bar{h}(v_2(x,t)),\\
      \partial_t v_2(x,t) =&d_2\mathcal {D}[v_2](x,t)-a_2v_2(x,t)+\bar{g}(v_1(x,t)).
  \end{aligned}
  \right.
\end{align}
where $t>0 $, $x\in\mathbb{R}$. According to \cite{Hsu-Yang},
we assume the nonlinearities $\bar{h}(\cdot)$ and $\bar{g}(\cdot)$ satisfy the following assumptions:
\begin{itemize}
\item [(B1)] $\bar{h},\bar{g}\in C^2(\mathbb{R}^+,\mathbb{R}^+)$, $\bar{h}(0)=\bar{g}(0)=0$, $K_2=\bar{g}(K_1)/a_2$, $\bar{h}(\bar{g}(K_1)/a_2)=a_1K_1$ and $\bar{h}(\bar{g}(u)/a_2)>a_1u$ for $u\in(0,K_1)$,
where $K_1$ is a positive constant.\medskip
\item [(B2)] $\bar{h}^\prime(0)\bar{g}^\prime(0)>a_1a_2$.\medskip
\item [(B3)] $\bar{h}^{\prime\prime}(v)\leq0$, $\bar{h}^\prime(v)\geq0$ for all $v\in[0,K_2]$ and $\bar{g}^{\prime\prime}(u)\leq 0$,
$\bar{g}^\prime(u)\geq0$ for all $u\in[0,K_1]$. \medskip
\item [(B4)] $\min\{a_1,a_2\}>\max\{\bar{g}^\prime(K_1),\bar{h}^\prime(K_2)\}$.
\end{itemize}
It's clear that \eqref{6.1} has two equilibria $(0,0)$ and $(K_1,K_2)$.
Under assumptions (B1)--(B3), the existence of traveling wave solutions for system \eqref{6.1} connecting $(0,0)$ and $(K_1,K_2)$ was proved by Hsu and Yang \cite{Hsu-Yang}.
Moreover, we can rewrite \eqref{6.1} in the form of \eqref{1.1} by setting
\begin{align*}
h(v_1(x,t),v_2(x,t))&:=-a_1v_1(x,t)+\bar{h}(v_2(x,t)),\\       g(v_1(x,t),v_2(x,t))&:=-a_2v_2(x,t)+\bar{g}(v_1(x,t)).
\end{align*}
Obviously, the assumptions {\rm (B1)--(B4)} imply that the conditions {\rm (H1)--(H4)} hold.
Therefore we can obtain same assertion of Theorems  \ref{thm2.1} and \ref{thm2.2} for system \eqref{6.1}.\medskip

Next, we illustrate some examples for $\bar{h}(x)$ and $\bar{g}(x)$ which satisfy the assumptions (B1)--(B4).

\begin{ex}{\rm Assume the Holling-II type functions
\begin{equation}\label{ex1}
\bar{h}(x)=\frac{\alpha_1 x}{\beta_1+\gamma_1 x}\ \mbox{and}\ \bar{g}(x)=\frac{\alpha_2 x}{\beta_2+\gamma_2 x},
\end{equation}
where $\alpha_i,\beta_i,\gamma_i$, $i=1,2$ are positive constants. Then $\bar{h}(0)=\bar{g}(0)=0$. Furthermore, elementary computations imply that
\begin{align*}
&(K_1,K_2)=\big(\frac{\alpha_1\alpha_2-a_1a_2\beta_1\beta_2}{a_1(a_2\beta_1\gamma_2+\alpha_2\gamma_1)},
\frac{\alpha_1\alpha_2-a_1a_2\beta_1\beta_2}{a_2(a_1\beta_2\gamma_1+\alpha_1\gamma_2)}\big),\medskip\\
&\bar{h}^\prime(0)=\frac{\alpha_1}{\beta_1},\ \bar{g}^\prime(0)= \frac{\alpha_2}{\beta_2},\ \bar{h}^{\prime\prime}(0)=\frac{-2\alpha_1\gamma_1}{\beta_1^2},\
\bar{g}^{\prime\prime}(0)=\frac{-2\alpha_2\gamma_2}{\beta_2^2},\medskip\\
&\bar{g}^\prime(K_1)=\frac{\alpha_2\beta_2}{(\beta_2+\gamma_2K_1)^2} \ \mbox{and}\ \bar{h}^\prime(K_2)=\frac{\alpha_1\beta_1}{(\beta_1+\gamma_1K_2)^2}.
\end{align*}
Hence, the assumptions (B1)--(B4) hold provided that
\begin{align}
\alpha_1\alpha_2>a_1a_2\beta_1\beta_2\ \mbox{and}\ \min\big\{a_1,a_2\}>\max\{\frac{\alpha_2\beta_2}{(\beta_2+\gamma_2K_1)^2} ,\frac{\alpha_1\beta_1}{(\beta_1+\gamma_1K_2)^2}\big\}.\label{ex}
\end{align}
Noting that the inequalities of \eqref{ex} hold when $\beta_1$ and $\beta_2$ are small enough. Thus, we can obtain the following result.
\begin{thm}\label{thm6.1}
Let $\bar{h}(x)$ and $\bar{g}(x)$ be the functions given by \eqref{ex1}, where $\alpha_i,\beta_i,\gamma_i$, $i=1,2$ are positive constants satisfying the conditions of \eqref{ex}.
Then the assertions of Theorems \ref{thm2.1} and \ref{thm2.2} hold for system \eqref{6.1}.
\end{thm}}
\end{ex}
\begin{ex}{\rm
Assume
\begin{equation}\label{ex2}
\bar{h}(x)=ax\ \mbox{and } \bar{g}(x)=px e^{-qx^m},
\end{equation}
where $p, q$ and $m$ are positive constants. The function $\bar{g}(x)$ is called the Ricker type function. Of particular, when $m=1$, $\bar{g}(x)$ is reduced to the Nicholson's blowflies function.
%
By elementary computations, we have
\begin{align}
(K_1,K_2)&=\big((\frac{1}{q}{\ln(\frac{ap}{a_1a_2})})^{{1}/{m}},\frac{a_1}{a} K_1\big),\label{equ}\medskip\\
\bar{g}^\prime(x)=p(1-qmx^m)e^{-qx^m}&\ \mbox{and }\  \bar{g}^{\prime\prime}(x)=pqmx^{m-1}e^{-qx^m}(qmx^m-1-m).\label{gd}
\end{align}
Let us set $u_*:=({mq})^{-{1}/{m}}$, then \eqref{gd} implies that $\bar{g}(x)$ is non-decreasing on $[0, u_*]$ and non-increasing on $[u_*, \infty]$. Therefore,
if $1<\frac{ap}{a_1a_2}\leq e^{{1}/{m}}$, we have
\begin{center}
$K_1\leq u_*$, $\bar{g}^\prime(x)\geq0$ for
$x\in[0,K_1]$ and $\bar{h}^\prime(x)=a\geq0$ for $x\in[0,K_2]$.
\end{center}
Furthermore, for $x\in[0,K_1]$, we know that
\begin{equation*}
\bar{g}^\prime(K_1)=\frac{a_1a_2}{a}(1-m\ln(\frac{ap}{a_1a_2}))\ge 0\ \mbox{and }\ \bar{g}^{\prime\prime}(x)\leq0 .
\end{equation*}
Hence, the assumptions (B1)--(B4) hold provided that
\begin{align}\label{ex22}
\min\{a_1,a_2\}>\max\big\{a ,\frac{a_1a_2}{a}(1-m\ln(\frac{ap}{a_1a_2}))\big\}.
\end{align}
Thus, we can obtain the following result.
\begin{thm}\label{thm6.2}
Let $\bar{h}(x)$ and $\bar{g}(x)$ be the functions given by \eqref{ex2}, where $p,q,m$, $i=1,2$ are positive constants. Assume $1<\frac{ap}{a_1a_2}\leq e^{{1}/{m}}$ and the conditions of \eqref{ex22} hold. Then the assertions of Theorems \ref{thm2.1} and \ref{thm2.2} hold for system \eqref{6.1}.
\end{thm}}
\end{ex}


\end{document}